\documentclass[12pt]{amsart}
\usepackage{geometry}
\usepackage{amssymb}
\usepackage{latexsym}
\usepackage{amsfonts}
\usepackage{amsmath}
\usepackage{amsthm}
\usepackage{xypic}
\usepackage{graphicx}
\usepackage{mathptmx}
\usepackage{url}
\newtheorem{thm}{Theorem}[section]

\newtheorem{prop}[thm]{Proposition}
\newtheorem{cor}[thm]{Corollary}
\newtheorem{definition}{Definition}[section]

\def\ord{\mathrm{ord}}
\def\min{\mathrm{min}}

\def\GL{\mathrm{GL}}

\def\PGL{\mathrm{PGL}}
\def\det{\mathrm{det}}

\def\a{\textbf{a}}

\author{Lucien Szpiro}
\address{Lucien Szpiro\\
        Ph.D. Program in Mathematics\\
        CUNY Graduate Center\\
        365 Fifth Avenue, New York, NY 10016-4309 U.S.A.}
\email{lszpiro@gc.cuny.edu}

\author{Michael Tepper}
\address{Michael Tepper\\
        Division of Science and Engineering\\
        Penn State Abington\\
        1600 Woodland Road\\
        Abington, PA 19001 U.S.A.}
\email{mlt16@psu.edu}

\author{Phillip Williams}
\address{Phillip Williams\\
        The King's College\\
        52 Broadway, 5th Fl, New York, NY 10004}
\email{pwilliams@tkc.edu}

\title[Semi-stable implies Minimal]{Semi-stable reduction implies minimality of the resultant}
\date{\today}

\begin{document}

\maketitle

\begin{abstract}
For a dynamical system on $\mathbb{P}^n$ over a number field or a function field, we show that semi-stable reduction implies the minimality of the resultant. We use this to show that every such dynamical system over a number field admits a globally minimal presentation. 
\end{abstract}

\section{Introduction}

Let $K$ be a number field or a function field of a complete nonsingular curve defined over an algebraically closed field $k$ and $$\varphi: \mathbb{P}^n \rightarrow \mathbb{P}^n$$ be a morphism of degree $d$, defined over $K$. Let $p$ be a closed point of the one dimensional scheme associated to $K$, which is either the spectrum of the ring of integers of a number field or the curve associated to the the function field.   One natural notion of good reduction at $p$ of a morphism is: when reducing the coefficients that define the morphism for a suitably chosen pair of homogeneous coordinates $[X_0,\dots,X_n]$ on $\mathbb{P}^n$, we get a map over the residue field of the same degree. This notion of good reduction is measured by the minimal resultant, defined below.

When considering good reduction for morphisms over global fields as above, we are interested in both local and global questions. One local question is: when does a morphism have good reduction at a particular $p$? Related to this is the more precise question: what is the order of vanishing of the minimal resultant at $p$? A natural global question to ask is the following: is it possible to find a choice of coordinates which realizes the minimal resultant at all points? This is akin to the minimal model question in the theory of curves. 

Both of these questions have been explored in \cite{SilvermanDynamics} and \cite{STW}. More recently, \cite{BruinMolnar} has provided useful algorithms for computing the minimal resultant locally, in the case of $\mathbb{P}^1$. 

In \cite{STW}, a GIT based criterion for minimality of the resultant is given. Specifically, we show that semi-stable reduction implies minimality for quadratic maps on $\mathbb{P}^1$ over a function field. The number field case is also addressed, but only proved for certain families of maps.

Here, we will further investigate both the local and the global question.  Strengthening the results of \cite{STW}, we prove that semi-stable reduction implies minimality of the resultant for maps on $\mathbb{P}^n$ of any degree over either a function field or a number field. More precisely, given a choice of coordinates on projective space for which the reduction of our morphism at point $p$ is semi-stable in the parameter space for rational maps, we show that this choice of coordinates also yields a minimal value for the order of vanishing of the resultant. This yields an easy test for minimality, and thus for good reduction, which should be useful for concrete examples. In addition, it allows one to determine \textit{potential} good reduction (good reduction after base extension) algorithmically, in the case of $\mathbb{P}^1$. This is due to Levy, who has shown in \cite{Levy} and \cite{Levy3}, the existence of a semi-stable presentation, after base extension, and has provided an algorithm that concretely computes it in the $\mathbb{P}^1$ case. 

In regards to the global questions, our main result, combined with the result of Levy \cite{Levy} mentioned above, gives us a way of finding, after base extension, a globally minimal choice of coordinates for a morphism in the number field setting.

\textit{Acknowledgements}: The authors would like to thank Alon Levy, Nikita Miasnikov, and Bart Van Steirteghem for helpful comments and discussions.
\section{Setup and Preliminaries}

In this section, we give our basic setup and give some preliminaries required for the results in the following sections.

As above, let $K$ be a number field or a function field  of a complete nonsingular curve defined over an algebraically closed field $k$, and let $$\varphi: \mathbb{P}^n \rightarrow \mathbb{P}^n$$ be a morphism of degree $d^n$, defined over $K$. Recall that morphisms of degree $d^n$ with respect to a fixed coordinate system $[X_0,\dots,X_n]$ on $\mathbb{P}^n$ are parameterized by an open subset $\mathrm{Hom}_d^n$ of points in the projective space $\mathbb{P}^N$ corresponding to the coefficients of the $n+1$ homogeneous polynomials of degree $d$, where $N= \binom{n+d}{d}(n+1)-1$ . This open set is $\mathbb{P}^N - V(\rho)$, where $\rho$ is the resultant of the $n+1$ polynomials. Thus, after a choice of $[X_0,\dots,X_n]$, our $\varphi$ corresponds to some point $[\mathbf{a}_0,\dots,\mathbf{a}_n]$ in projective space $\mathbb{P}^N$. Here $\mathbf{a}_i = a_{i,0}, \dots, a_{i,D}$, where $D=\binom{n+d}{d}$. We call this point $[\mathbf{a}_0,\dots,\mathbf{a}_n]$ a \textbf{presentation} of $\varphi$ (with respect to $[X_0,\dots,X_n]$). The choice of $[X_0,\dots,X_n]$ determines the presentation, but the coordinates $\mathbf{a}_0,\dots, \mathbf{a}_n$ are determined only up to scalar multiple. 

In general (over any ring or field), the different choices of $[X_0,\dots,X_n]$ induce an action of the group $\mathrm{PGL}_{n+1}$ on $\mathbb{P}^N$. For $[\Gamma]$ in $\mathrm{PGL}_{n+1}(K)$, we write $[\mathbf{a}_0^{\Gamma},\dots,\mathbf{a}_n^{\Gamma}]$ for the new coefficients under the action of $[\Gamma]$. The $\mathbf{a}_0^{\Gamma},\dots,\mathbf{a}_n^{\Gamma}$ are obtained from $\mathbf{a}_0,\dots,\mathbf{a}_n$ by precomposing $\varphi$ written with respect to $[X_0,\dots,X_n]$ with $\Gamma$, and then post-composing with the adjoint of $\Gamma$. This is not very practical to explicitly state, especially for large $d$, but each new coefficient is a just homogeneous polynomial in the variables $\mathbf{a}_0,\dots,\mathbf{a}_n$, and the entries $\alpha_{i,j}$ of the $(n+1)\times(n+1)$ matrix $\Gamma = \left(\alpha_{i,j}\right)$. 

Given $p$, a point of $C = \mathrm{Spec}(O_K)$ or a point of the curve $C$ (where $K = k(C)$, $k$ algebraically closed) for the number field case and function field case respectively, we can choose $(\mathbf{a}_0,\dots,\mathbf{a}_n)$ to be \textit{normalized} at $p$. This means $\mathbf{a}_0,\dots,\mathbf{a}_n$ are in the local ring at $p$ and hence do not all vanish when reduced modulo $p$. The reduction modulo $p$ is written as $[\mathbf{a}_0(p),\dots,\mathbf{a}_n(p)]$, where $(\mathbf{a}_0,\dots,\mathbf{a}_n)$ are taken to be normalized at $p$; under this restriction, it is a well defined point of projective space.  In addition, $[\mathbf{a}_0(p),\dots,\mathbf{a}_n(p)]$ will correspond to a morphism over $\kappa (p)$ if and only if the resultant $\rho(\mathbf{a}_0(p), \dots,\mathbf{a}_n(p)) \neq 0$. 

The reduction so defined depends on $[X_0,\dots,X_n]$, and so the question of whether the reduction corresponds to a morphism over the residue field also depends on this choice. We can eliminate this dependence by considering all possible choices of coordinates on $\mathbb{P}^n$.

\begin{definition}We say that $\varphi$ has \textbf{good reduction} at $p$ if there exists a choice of coordinates $[X_0,\dots,X_n]$ and $(\mathbf{a}_0,\dots,\mathbf{a}_n)$ normalized such that $\rho(\mathbf{a}_0(p),\dots,\mathbf{a}_n(p)) \neq 0$. The morphism $\varphi$ has \textbf{bad reduction} otherwise. 
\end{definition}

This is all measured by the minimal resultant: 

\begin{definition}
Given a choice of coordinates $[X_0,\dots,X_n]$ and corresponding normalized $(\mathbf{a}_0,\dots,\mathbf{a}_n)$ such that $[\mathbf{a}_0,\dots,\mathbf{a}_n]$ is a presentation of $\varphi$ with respect to $[X_0,\dots,X_n]$, define $(R_{\varphi, [X_0,\dots,X_n]})_p := \mathrm{ord}_p(\rho(\mathbf{a}_0,\dots\mathbf{a}_n))$.  The \textbf{minimal resultant} is then the following divisor: 
$$R_{\varphi} = \sum_p \underset{[X_0,\dots,X_n]}{\mathrm{min}} \{(R_{\varphi, [X_0,\dots,X_n]})_p\}[p]$$ 
\end{definition}
The minimal resultant vanishes precisely at the points of good reduction. 

Sometimes it is more natural to consider whether a morphism has good reduction after an algebraic extension. 

\begin{definition} We say that $\varphi$ has \textbf{potential good reduction} at $p$ if there exists an algebraic extension $K'$ of $K$ such that $\varphi$, considered as a morphism over $K'$, has good reduction at some $p'$ lying over $p$. 
\end{definition}

There is an open subscheme $(\mathbb{P}^N)^{ss} \subset \mathbb{P}^N$ which is ``nice" in the sense that it is possible to define a GIT categorical quotient $\mathcal{M}_d^{n,ss}$ with respect to these points under the group action described above. It contains $\mathrm{Hom}^n_d$. In general the GIT semi-stable points are, intuitively, the points for which one cannot ``push the point to $0$" using the group action. Relatedly, every point in the semi-stable space is non-vanishing on some homogeneous form for which the group in question is invariant. This is the key property we will use in the proof of our main result. For more on the GIT construction of the moduli space in question, see \cite{SilvermanSpace} for the $n=1$ case, and \cite{PST} and \cite {Levy2} for the general case. 


We will need the following concrete characterizations of the moduli space and its compactification, which we may take throughout as their definitions:
\begin{prop}\label{defineMd}
The space $\mathcal{M}_d^n$ is isomorphic to $\mathrm{Spec}((A_d)_{(\rho)}^{\mathrm{SL}(n+1)})$, and the space $\mathcal{M}_d^{n, ss}$ is isomorphic to $\mathrm{Proj}((A_d)_{\rho}^{\mathrm{SL}(n+1)})$, where $A_d=\mathbb{Z}[\mathbf{a}_0,\dots,\mathbf{a}_n]$ and $\mathbf{a}_0,\dots,\mathbf{a}_n$ are each $\binom{n+d}{d}$ indeterminants.
\end{prop}
\begin{proof}
The fact that  $\mathcal{M}_d^n$ and $\mathcal{M}_d^{n, ss}$ are geometric and categorical quotients respectively is established in  \cite{Levy2}, \cite{PST} and \cite{SilvermanSpace}. This explicit description follows for such quotients in the setting we are in.  The $n=1$ case is mentioned in \cite{SilvermanSpace}, which follows just as well for arbitrary $n$ based on the results in \cite{Seshadri}. In particular, see Theorem 4 and the following remark 8. 
\end{proof}

We will also need a generalization to $\mathbb{P}^n$ of some basic results about how the order of vanishing of the resultant is affected by conjugation. The following is a straightforward generalization of Proposition 4.95 in \cite{SilvermanDynamics}: 
\begin{prop}\label{prop495}
Let $\varphi:\mathbb{P}^n\rightarrow\mathbb{P}^n$ be a degree $d^n$ endomorphism with a presentation $[\mathbf{a}_0,\dots,\mathbf{a}_n]$ and $p$ a point of $C = \mathrm{Spec}(O_K)$ or a point of the curve $C$ as before, where $K = k(C)$.
\begin{quote}
  
\textbf{(a)}    The valuation of the resultant of $\varphi$ is given by the formula
$$\ord_p(R_\varphi)=\ord_p(\rho(\a_0,\dots,\a_n))-(n+1)d^n\min \{\ord_p(\a_0),\dots,\ord_p(\a_n)\}.$$
Here $\ord_p(\a_i)$ is the minimal order of vanishing at $p$ of any of the coordinates of the tuple $\a_i$. Note that normalized coefficients are not assumed in this statement (indeed the statement is trivial in that case).

\textbf{(b)} Let $\Gamma \in\GL_{n+1}(K)$.  Then
$$\ord_p(\rho(\a_0^\Gamma,\dots,\a_n^\Gamma))=\ord_p(\rho(\a_0,\dots,\a_n))+(n+d)d^n\ord_p(\det\ \Gamma),$$
$$\min\{\ord_p(\a_0^\Gamma),\dots,\ord_p(\a_n^\Gamma)\}\geq \min\{\ord_p(\a_0),\dots,\ord_p(\a_n)\}+(d+1)\ord_p(\Gamma).$$

\textbf{(c)} If $U\in\GL_{n+1}(O_{K,p})$, then
$$\ord_p(\rho(\a_0^U,\dots,\a_n^U))=\ord_p(\rho(\a_0,\dots,\a_n)),$$
$$\min\{\ord_p(\a_0^U),\dots,\ord_p(\a_n^U)= \min\{\ord_p(\a_0),\dots,\ord_p(\a_n)\}.$$
\end{quote}
\end{prop}
The proof is exactly the same as in \cite{SilvermanDynamics} with a change in regards to the coefficients related to the degree.


\section{Semi-stability implies minimality}

Here, we show semi-stability implies minimality. To state the result, we make the following definitions:

\begin{definition} A presentation $[\mathbf{a}_0,\dots,\mathbf{a}_n]$ of $\varphi$ is \textbf{semi-stable} at $p$ if the reduction $[\mathbf{a}_0(p),\dots,\mathbf{a}_n(p)]  \in (\mathbb{P}^{N})^{ss}(\kappa(p))$. 
\end{definition}

\begin{definition} A presentation $[\mathbf{a}_0,\dots,\mathbf{a}_n]$ of $\varphi$ is \textbf{minimal} at $p$ if $(R_{\varphi, [X_0,\dots,X_n]})_p$ is minimal with respect to all choices of coordinates on $\mathbb{P}^n(K)$. 
\end{definition}

In \cite{STW}, the following are shown:

\begin{prop}
Let $K$ be a function field, and $p \in C$. Let $\varphi:\mathbb{P}^1\rightarrow\mathbb{P}^1$ be a morphism of degree $2$. Let $[\mathbf{a},\mathbf{b}]$ be a presentation of $\varphi$. If $[\mathbf{a},\mathbf{b}]$ is semi-stable at $p$, then $[\mathbf{a},\mathbf{b}]$ is minimal at $p$. 
\end{prop}

\begin{prop}
Let $K$ be a function field or number field, and $p \in C$. Let $\varphi:\mathbb{P}^1\rightarrow\mathbb{P}^1$ be morphism of degree $d$. Let $[\mathbf{a},\mathbf{b}]$ be a presentation of $\varphi$. If $[\mathbf{a},\mathbf{b}]$ is semi-stable at $p$, and is not in $\mathrm{Rat}_d(\kappa(p))$,  then $\varphi$ has bad reduction at $p$. 
\end{prop}

We will now show the following generalization that implies both propositions: 

\begin{thm}
Let $K$ be a function field or a number field, and $p \in C$. Let $\varphi$ be a morphism of degree $d$. Let $[\mathbf{a}_0,\dots,\mathbf{a}_n]$ be a presentation of $\varphi$. If $[\mathbf{a}_0,\dots,\mathbf{a}_n]$ is semi-stable at $p$, then $[\mathbf{a}_0,\dots,\mathbf{a}_n]$ is minimal at $p$.
\end{thm}

\begin{proof}
Suppose we have a choice of coordinates $[X_0,\dots,X_n]$ and a corresponding presentation $[\mathbf{a}_0,\dots,\mathbf{a}_n]$ that is semi-stable at $p$, where we assume $(\mathbf{a}_0,\dots,\mathbf{a}_n)$ have been normalized at p. Recall that $\mathcal{M}_d^n$ is the moduli space of degree $d$ maps on $\mathbb{P}^n$ and $\mathcal{M}_d^n$ is an affine scheme over $\mathbb{Z}$ and $\varphi \in \mathcal{M}_d^n(K)$. Now $\mathcal{M}_d^{n,ss}$ is the quotient of the semi-stable space mentioned above.  It contains $\mathcal{M}_d^n$, and is a projective scheme over $\mathbb{Z}$. We write these schemes concretely: $\mathcal{M}_d^n = \mathrm{Spec}((A_d)_{(\rho)}^{\mathrm{SL}(n+1)})$ and $\mathcal{M}_d^{n,ss} = \mathrm{Proj} (A_d^{\mathrm{SL}(n+1)})$, where $A_d = \mathbb{Z}[\mathbf{a}_0,\dots,\mathbf{a}_n]$ as in Proposition \ref{defineMd}. Consider the image $x$ of the reduced point $[\mathbf{a}_0(p),\dots, \mathbf{a}_n(p)]$ in $\mathcal{M}_d^{n,ss}$ under the natural map:
$$ \pi: (\mathbb{P}^N)^{ss} \rightarrow \mathcal{M}_d^{n,ss}.$$
This $x$ corresponds to a homogeneous prime ideal $\mathfrak{p}_x$ of $A_d^{\mathrm{SL}(n+1)}$. Hence we can find an element $f \notin \mathfrak{p}_x$ where $f$ is homogeneous of degree $m > 0$. The element $f$ is a homogeneous $\mathrm{SL}_{n+1}$ invariant polynomial in $(\mathbf{a}_0,\dots,\mathbf{a}_n)$ and has the property that $f(\mathbf{a}_0(p),\dots, \mathbf{a}_n(p)) \neq 0$. It follows that $\mathrm{ord}_p(f(\mathbf{a}_0,\dots, \mathbf{a}_n)) = 0$. The $\mathrm{SL}_{n+1}$ invariance is over $\mathbb{Z}$, and so is valid over any field, since the ring of invariants is obtained by base extension. Now, replace $f$ by $f^\delta$, where $\delta=(n+1)d^n$ is the degree of the resultant.  All that has been stated about $f$ still applies and now $f$ is degree $\delta m$. Set $\sigma = \frac{f}{\rho^m}$. Then $\sigma$ is of total degree zero (in fact, $\sigma$ is a global section of $\mathcal{M}_d^n$).  Further, $\sigma$ is actually $\mathrm{GL}_{n+1}$ invariant, being the ratio of two homogeneous $\mathrm{SL}_{n+1}$ invariant functions. This implies the $\PGL_{n+1}$ action is well defined on $\sigma$ and $\sigma$ is invariant with respect to this action as well. Choose $[\Gamma] \in \mathrm{PGL}_{n+1}(K)$ such that $[\mathbf{a}_0^{\Gamma},\dots, \mathbf{a}_n^{\Gamma}]$ is minimal at $p$. Let $S = (R_{\varphi})_p.$ Let $S' = (R_{[X_0,\dots,X_n]})_p$. Our goal is to show that $S = S'$. Let $\Gamma$ have coordinates all in the local ring at $p$. Then $(\mathbf{a}_0^{\Gamma}\dots, \mathbf{a}_n^{\Gamma})$ has coordinates in the local ring and we normalize: $(\mathbf{a}_0',\dots, \mathbf{a}_n') = (c\mathbf{a}_0^{\Gamma},\dots, c\mathbf{a}_n^{\Gamma})$. Everything is still in the local ring at $p$, therefore $\mathrm{ord}_p( f(\mathbf{a}_0',\dots, \mathbf{a}_n')) := r >0$. Thus $\mathrm{ord}_p(\sigma(\mathbf{a}_0',\dots, \mathbf{b}_n')) = \mathrm{ord}_p(\frac{f(\mathbf{a}_0',\dots, \mathbf{a}_n')}{\rho(\mathbf{a}_0',\dots, \mathbf{a}_n')^m}) = r - mS$. On the other hand, by $\mathrm{PGL}_{n+1}(K)$ invariance, we also have $\mathrm{ord}_p(\sigma(\mathbf{a}_0',\dots, \mathbf{a}_n')) = \mathrm{ord}_p(\sigma(\mathbf{a}_0, \dots, \mathbf{a}_n))= \mathrm{ord}_p(\frac{f(\mathbf{a}_0,\dots, \mathbf{a}_n)}{\rho^m(\mathbf{a}_0,\dots, \mathbf{a}_n)}) = -mS'$. Thus:
\begin{align*}
 r-mS = -mS' &\\
  -mS \leq -mS' &\\
 \end{align*}
Hence $   S \geq S'$ and therefore $S = S'$.
\end{proof}

A corollary is that semi-stable bad reduction implies that bad reduction will continue to be present in all further base extensions.

\begin{cor}
\label{PGR}
Suppose $\varphi$ has a semi-stable presentation $[\mathbf{a}_0,\dots, \mathbf{a}_n]$ at $p$ with respect to $[X_0,\dots,X_n]$. Then $\varphi$ has good reduction at $p$ if and only if $\varphi$ has potential good reduction at $p$. 
\end{cor}
\begin{proof}
If $\varphi$ has bad reduction at $p$, then the order of vanishing of the minimal resultant is positive, and is realized by the presentation $[\mathbf{a}_0,\dots,\mathbf{a}_n]$. Upon base extension, the order of vanishing of the resultant for the presentation $[\mathbf{a}_0,\dots, \mathbf{a}_n]$ clearly remains positive, and the presentation remains semi-stable, and is therefore minimal. 
\end{proof}

\subsection{The Minimal presentation question}
In this section, we consider the question of whether it is always possible to find a global minimal presentation, answering the affirmative for the number field case. 

If one could find a global semi-stable presentation, it would follow, from the above, that this presentation is also minimal. As we have already noted, in \cite{Levy}, A. Levy shows that it is always possible, after base extension, to find a semi-stable presentation locally. The question of whether this can be done globally, however, is more difficult. Levy observes that this question is equivalent to that of determining the triviality of a vector bundle that is constructed naturally from the different choices of coordinates required to write a semi-stable presentation for each point. He gives counterexamples which demonstrate that it is not always possible to trivialize this vector bundle. This does not show that a minimal presentation is impossible, but it does mean that we can't always use semi-stability to find one. In the number field case, however, this is always possible:

\begin{thm}
Let $K$ be a number field, and $\varphi:\mathbb{P}^n\rightarrow\mathbb{P}^n$ a morphism. Then there is an algebraic extension $L$ of $K$ such that $\varphi$, considered as a morphism over $L$, has a minimal presentation.
\end{thm}
\begin{proof}
In \cite{Levy}, it is shown that a semi-stable presentation is possible everywhere locally, after a finite algebraic extension. Since there are only finitely many points of bad reduction, this implies that we can find an algebraic extension $K'$ for which every point has a semi-stable presentation. Considering all of the required choices of coordinates for these presentations together, one gets a vector bundle over $\mathrm{Spec}(\mathcal{O}_{K'})$. 

One can trivialize any line bundle over $\mathrm{Spec}(\mathcal{O}_{K'})$, after an appropriate finite degree base extension. This can be shown using class field theory (the principal ideal theorem), or more directly using the finiteness of the class group. Given a locally free sheaf $\mathcal{L'}$, consider the isomorphism $\mathcal{L'}^m \cong \mathcal{O}_{\mathrm{Spec}(\mathcal{O}_{K'})}$, which exists for some $m$. Introducing the appropriate $m$-th roots, this isomorphism will induce an isomorphism  $\mathcal{L''} \cong \mathcal{O}_{\mathrm{Spec}(\mathcal{O}_{L})}$, where $L$ is the field formed by adjoining these elements, and $\mathcal{L''} = \mathcal{L'} \otimes_{\mathcal{O}_{K'}} \mathcal{O}_{L}$.

For a vector bundle over $\mathrm{Spec}(\mathcal{O}_{K'})$ of arbitrary finite dimension $n$, consider the associated projective module. A standard result in commutative algebra says that this will be of the form $\mathcal{O}_{K'}^{n-1} \oplus I$, where $I$ is an ideal of $\mathcal{O}_{K'}$. Trivializing $I$ as above, we see that vector bundles of arbitrary finite dimension can be trivialized after finite base extension. 

Thus, returning to the vector bundle that arises from our situation, we pass to an algebraic extension $L$ which trivializes this vector bundle. We now have a (finite) collection of affine open sets $U_i = \mathrm{Spec}(R_i)$ in $\mathrm{Spec}(\mathcal{O}_{L})$. For each $i$, there exists a choice of coordinates $(X_{0,i},\dots,X_{n,i})$ on $\mathbb{A}^{N+1}_{R_i}$ such that the morphism $\varphi$ (now thought of as over $L$) has semi-stable reduction with respect to the associated projective coordinates $[X_{0,i},\dots,X_{n,i}]$ on $\mathbb{P}^N_L$.  The triviality of this vector bundle implies that there is a single global choice of coordinates $(X_0,\dots,X_n)$ on $\mathbb{A}^{N+1}_{\mathcal{O}_L}$ such that, on each $U_i$, this global choice of coordinates is related to $(X_{0,i},\dots, Y_{n,i})$ by a matrix $\Gamma_i \in \mathrm{GL}_{n+1}(R_i)$. Note that this matrix is normalized and has good reduction at every point in $U_i$. 

For any $p \in \mathrm{Spec}(\mathcal{O}_L)$, chose normalized coordinates $(\mathbf{a}_0,\dots \mathbf{a}_n)$. In general, acting on coordinates $(\mathbf{a}_0,\dots, \mathbf{a}_n)$ by a matrix whose coordinates are normalized at $p$, and which has good reduction at a point $p$, will preserve the property of the coordinates $(\mathbf{a}_0, \dots,\mathbf{a}_n)$ being normalized, by proposition \ref{prop495}.
 It is clear from this that the reductions of $[\mathbf{a}_0,\dots, \mathbf{a}_n]$ and $[\mathbf{a}_0^{\Gamma_i},\dots, \mathbf{a}_n^{\Gamma_i}]$ at a point $p \in U_i$ are conjugate over the residue field at $p$, and thus the property of being semi-stable at $p$ is preserved by this conjugation. From this it follows that we have found a global semi-stable presentation over the field $L$.
\end{proof}

Now, Levy provides counterexamples that show a trivial vector bundle is not always possible in the function field setting. These examples thus also show it is not always possible to find a globally semi-stable presentation. However, the question of a globally minimal presentation is still open in these cases. It might be interesting to examine the vector bundle associated to the minimal presentations given by the algorithm in \cite{BruinMolnar} in these examples.

\end{document}